\documentclass{amsart}
\usepackage{latexsym}
\usepackage{amsmath}
\usepackage{amsthm}
\usepackage{amssymb}
\title[Conjugacy classes and ordinary characters of the Monster]
{Verification of the conjugacy classes and ordinary character table of the Monster}
\author{Thomas Breuer, Kay Magaard and Robert A. Wilson}
\date{19th October 2017; revised 21st March 2021; this version 12th December 2024}
\begin{document}
\newcommand{\Atlas}{$\mathbb{ATLAS}$}
\def\Irr{\operatorname{Irr}}
\newcommand{\M}{\mathbb M}
\newcommand{\B}{\mathbb B}
\newcommand{\Fi}{\mathrm{Fi}}
\newcommand{\HN}{\mathrm{HN}}
\newcommand{\Th}{\mathrm{Th}}
\newcommand{\Co}{\mathrm{Co}}
\newcommand{\McL}{\mathrm{McL}}
\newcommand{\Suz}{\mathrm{Suz}}
\newcommand{\Ja}{\mathrm J}
\newcommand{\Ma}{\mathrm M}
\newcommand{\He}{\mathrm{He}}

\newcommand\pr{\noindent {\rm Proof.\ \ \/}}
\newcommand\eop{{\hfill $\Box$ \vskip 2ex}}
\newtheorem{lemma}{Lemma}
\newtheorem{hypo}{Hypothesis}
\newtheorem{fact}{Fact}
\newtheorem{prop}{Proposition}

\begin{abstract}
As part of the programme to re-compute the character tables of all the groups in
the \Atlas\ we re-compute the character table of $\M$, the 
Monster simple group. We operate under the uniqueness hypotheses of $\M$  and the 
existence of an ordinary faithful representation of 
degree $196883 = 47.59.71$ and determine 
the conjugacy classes and centralizer orders of the elements of $\M$.
Along the way we re-compute the character tables 
of centralizers of $p$-elements for $p < 11$ as well as 
fusions of conjugacy classes of these centralizers in $\M$. 

\end{abstract}

\maketitle

\centerline{\textit{In memory of Kay Magaard and Richard Parker}}

\section{Introduction}
The character table of the Monster sporadic simple group was originally calculated by Fischer and Livingstone
using computer programs written by Thorne. The table was checked in various ways before publication in the
\Atlas\ \cite{Atlas} in 1985, and its incorporation into GAP \cite{GAP}. 
Nevertheless, the problem of reliability and reproducibility of this table remains.
The other character tables in the \Atlas\ have now all been independently checked \cite{verify,Bverify1,Bverify2,2Bverify},
leaving just this one outstanding case.

In this paper we outline the proof that the character table of the Monster is correct,
and describe the strategy for the computations, which are described in greater detail in \cite{steps}. The latter is intended to
provide enough detail to enable the reader to check every step of the proof, in order to satisfy the demands of reproducibility.

In Section~\ref{hypoth} we set out our basic assumptions, and summarise the results from the literature that we rely on.
In Section~\ref{conseq} we compute the restrictions of the character of degree $196883$ to several subgroups
that will be used later, and also compute the restriction to the double cover of the Baby Monster of the
permutation character on the $6$-transpositions. 
This calculation enables us to compute many centralizer orders using Frobenius reciprocity.

The main part of the paper is Section~\ref{ccls}, in which we compute the complete list of
conjugacy classes. Here the strategy is to begin with the classes of elements of even order, then those of order
divisible by $3$, then $5$, and the remaining primes in the order
$$11, 17, 19, 23, 31, 47, 13, 29, 41, 59, 71, 7.$$
Finally, in Section~\ref{chartab} we briefly describe the computation of the values of the character of degree $196883$,
and the computation of the full character table from that information.

\section{Basic hypotheses} 
\label{hypoth}
Our starting point is the existence and uniqueness of $\M$ as demonstrated
in \cite{G}, \cite{C} and \cite{GMS}. 
Precisely we make the following hypothesis.
\begin{hypo}
$\M$ is a finite group which possesses
\begin{enumerate}
 \item  precisely two classes of involutions, represented by $z$ and $t$,  with centralizers $C_\M(z) \cong 2.\B$, 
 the double cover of the Baby Monster,
and $C_\M(t) \cong 2^{1+24}.\Co_1$, respectively, and
\item an ordinary irreducible representation affording a character $\chi$ of degree
$196883$. 
\end{enumerate}
\end{hypo}
We label the involution classes so that $2A$ is represented by $z$ and $2B$ is represented by $t$. 

From the uniqueness paper \cite{GMS} we use the following facts: 
\begin{fact}
The suborbit lengths of $2.\B$ acting on $z^\M$, where $1 \neq z \in Z(2.\B)$, and the
corresponding point 
stabilizers in $2.\B$ are as given in Table~\ref{suborbits}. 
\end{fact}
\begin{table}
\caption{\label{suborbits}Suborbit information}
$$\begin{array}{cccr}
\mbox{Product class} &\mbox{Class representative} & \mbox{Stabilizer} & \mbox{Orbit length}\cr
\hline
1A && 2.\B & 1\cr
2A &z& 2^2.{}^2E_6(2) & 27143910000\cr
2B &t& 2^{2+22}\Co_2 & 11707448673375\cr
3A &x_1& \Fi_{23} &  2031941058560000
\cr
3C &x_3& \Th &91569524834304000\cr
4A && 2^{1+22}\McL &1102935324621312000\cr
4B && 2.F_4(2) & 1254793905192960000\cr
5A &y_1& \HN & 30434513446055706624\cr
6A && 2.\Fi_{22} & 64353605265653760000\cr
\hline
\end{array}$$
\end{table}

In particular we know that there exist elements $x_1$, $x_3$ and $y_1$ of orders $3$, $3$ and $5$ respectively 
with centralizers equal to $3.\Fi_{24}'$, $3 \times \Th$ and $5 \times \HN$. By construction we 
deduce that the $3$-elements $x_1$ and $x_3$ are rational and the normalizers are $3.\Fi_{24}$ and $S_3 \times \Th$ respectively.
The element $y_1$ of order $5$ fuses into $2.\B$ and thus we can deduce that it is rational.

We assume in addition the character tables of $2.\B$ and a number of other \Atlas\ groups that have already been verified
 \cite{Bverify1,Bverify2,2Bverify,verify}.
We use explicit representations for several $p$-local subgroups of the Monster
    (or groups such that some factor group is a subgroup of the Monster),
    from which we recompute their character tables. We believe the correctness of these
    representations is adequately justified in the literature. 
 Finally, we use (and sketch a proof of) the existence of a subgroup of the structure $3^8.O_8^-(3)$ from \cite{W},
    which is needed in the argument to justify the structure of the Sylow $41$-normalizer,
    mentioned in Section~\ref{3elements} in Step 4, and used in Section~\ref{29-71elements} in Step 2.

\section{First consequences}
\label{conseq}
We begin by showing that there are unique possibilities for the restrictions of $\chi$ to the subgroups $2.\B$ and $3.\Fi_{24}$. Full details of the calculations
are in Section 2 of \cite{steps}.
\begin{lemma}
 The restriction of $\chi$ to $C_\M(z)$ is $1a + 4371a+ 96255a + 96256a$. 
\end{lemma}
\begin{proof}
If $\psi \in \Irr(C_\M(z))$ with $\psi(1) \le \chi(1)$, then $\psi \in \{1a,4371a,96255a, 96256a\}$. 
The only integral linear combination (other than $ \chi(1)1a$) of these characters having degree equal to $\chi(1)$, and 
at most two distinct values on involution classes, is our claimed combination.  
\end{proof}

    \begin{lemma}
    \label{3F24char}
    The restriction of $\chi$ to $C_\M(x_1)$ is
    $1a + 8671a + 57477a + 
    783ab+64584ab$. 
    The restriction of $\chi$ to $N_\M(\langle x_1\rangle)$ is
    $1a+8671b+57477a+1566a+129168a$.
    \end{lemma}

    \begin{proof}
    If $\psi\in\Irr(C_\M(x_1))$ with $\psi(1)\le\chi(1)$,
    then $$\psi\in\{ 1a, 
    8671a, 
    57477a, 
    783a, 783b, 64584a, 64584b\}.$$
    The only integral linear combination (other than $\chi(1) 1a$)
    of these characters having degree equal to $\chi(1)$,
    and at most two distinct values on involution classes,
    is our claimed combination. Finally, the extensions to the normalizer are determined by
    character values on involutions.
    \end{proof}

Also we can deduce enough from the suborbit information (Table~\ref{suborbits}) to compute the centralizer orders of those elements 
of orders $3$, $5$, $7$, $11$, $13$, $17$, $19$, $23$, $31$ and $47$ that fuse into $2.\B$, by computing the
values of the induced character. 
Indeed, 
we can 
compute the full permutation character (restricted to $2.\B$) of $\M$ on the cosets of $2.\B$, one suborbit at a time.
This is done by GAP calculations, described in Section 3 of \cite{steps}, with the character table of $2.\B$, and restrictions to the subgroups
in Table~\ref{suborbits}.

For the four smallest non-trivial orbits we list the GAP/$\mathbb{ATLAS}$ numbers of the characters of $2.\B$ that occur
in the permutation character. The first three are multiplicity-free:
\begin{enumerate}
\item Product $2A$. 
The character of $C_\M(z)$ is 
$ [ 1, 2, 3, 5, 7, 13, 15, 17 ]$, that is \\
$\chi_1 + \chi_2 + \chi_3 +\chi_5+\chi_7 +\chi_{13}+\chi_{15} + \chi_{17}$.

\item Product $2B$. 
The character of $C_\M(z)$ is 
 $ [ 1, 3, 5, 8, 13, 15, 28, 30, 37, 40 ]$.   

 \item Product $3A$. 
  The character of $C_\M(z)$ is\\
$ [ 1, 2, 3, 5, 7, 8, 9, 12, 13, 15, 17, 23, 27, 30, 32, 40, 41, 54, 63, 68,\\ 77, 81, 83,
185, 186, 187, 188, 189, 194, 195, 196, 203, 208, 
  220 ]$. 
  
\item Product $3C$. 
The character of $C_\M(z)$ is \\
$[ 1, 3, 7, 8, 12, 13, 16, 19, 27, 28, 34, 38, 41, 57, 62, 62, 68, 70, 77, 78, 85,\\ 
89, 113, 114, 116, 129, 133, 142, 143, 145, 155, 156, 185, 187, 
  188, 193,\\ 
  195, 196, 201, 208, 216, 219,
   225, 232, 233, 235, 236, 237, 242 ]$ \\
   (so that $\chi_{62}$ occurs with multiplicity two).
\end{enumerate}
For the four largest orbits, we list in Table~\ref{suborbitchars}
the multiplicities of all the irreducibles of $2.\B$ in the permutation character,
in GAP/\Atlas\ order. In the first two cases, these are arranged $30$ per line, and in the last two, $20$ per line.

\begin{table}\caption{\label{suborbitchars}Characters of the large suborbits restricted to $C_\M(z)$} 
Product $4A$. 

$ \begin{array}{l}
[  1, 1, 2, 1, 2, 0, 2, 3, 2, 0, 0, 1, 4, 1, 2, 0, 3, 2, 0, 2, 0, 0, 2, 2, 0, 0, 2, 3, 1, 5,\\
       0, 4, 3, 2, 0, 0, 3, 2, 0, 6, 4, 0, 1, 1, 0, 0, 0, 0,  3, 0, 1, 0, 0, 5, 0, 5, 2, 0, 0, 2,\\
       0, 0, 4, 1, 0, 2, 0, 4, 2, 4, 4, 3, 0, 2, 4, 2, 4, 0, 3, 0, 3, 2, 5, 0, 1, 0, 3, 1, 0, 1,\\
       1, 2, 5, 3, 1, 1,  4, 5, 1, 1, 0, 3, 0, 0, 3, 2, 1, 1, 2, 1, 1, 4, 0, 3, 2, 3, 1, 3, 0, 1,\\
      3, 0, 2, 2, 1, 3, 3, 0, 0, 2, 0, 0, 0, 0, 3, 0, 3, 3, 3, 1, 0, 3, 0, 4, 0, 1, 0, 0, 2, 0,\\
       0, 2, 0, 0, 2, 1, 1, 0, 0, 0, 0, 1, 2, 1, 1, 1, 0, 1, 1, 1, 1, 1, 1, 0, 2, 1, 1, 3, 3, 0,\\
       0, 0, 1, 1, 1, 1, 2, 3, 2, 0, 0, 2, 2, 4, 3, 5, 2, 4, 0, 0, 0, 0, 5, 2, 0, 0, 0, 1, 1, 0,\\
       0, 0, 0, 0, 0, 7,  0, 0, 1, 7, 7, 0, 0, 0, 1, 6, 4, 5, 0, 0, 3, 0, 0, 0, 0, 0, 4, 1, 1, 3,\\
      8, 3, 2, 2, 5, 0, 1 ] 
      \end{array} $

Product $4B$.  

$ \begin{array}{l}
 [ 1, 1, 2, 0, 2, 0, 2, 2, 1, 0, 0, 2, 4, 1, 3, 0, 2, 1, 0, 0, 0, 0, 2, 1, 0, 0, 2, 2, 1, 4,\\
 0, 2, 1, 2,  0, 0, 3, 2, 0, 4, 4, 0, 0, 0, 0, 0, 1, 0, 0, 0, 1, 0, 0, 2, 1, 3, 3, 0, 0, 3,\\
     0, 1, 4, 0, 0, 3, 0, 6, 0, 3, 2, 0, 0, 1, 4, 1, 4, 2, 6, 1,  4, 0, 4, 0, 1, 1, 2, 0, 0, 
     3,\\ 2, 1, 3, 2, 0, 0, 4, 5, 3, 1, 0, 3, 0, 0, 1, 1, 2, 0, 0, 2, 0, 2, 0, 3,   3, 3, 0, 4,
     1, 0,\\ 4, 1, 1, 1, 1, 1, 2, 1, 1, 2, 3, 0, 0, 2, 2, 0, 5, 5, 3, 0, 1, 5, 1, 4, 0, 1, 0,
     1, 1, 0,\\ 0, 3, 1, 0, 2, 3, 1, 0, 2, 0, 0, 2, 1, 0, 1, 0, 0, 1, 0, 0, 0, 1, 0, 0, 1, 2,
     1, 4, 4, 0,\\ 0, 0, 3, 1, 1, 1, 2, 2, 2, 0, 0, 1, 2, 3, 3, 3, 1, 2, 0, 0, 1, 1, 4, 2, 0,
     0, 0, 3, 2, 0,\\ 0, 0, 0, 0, 0, 4, 0, 0, 1, 5, 5, 0, 1, 1, 2, 2, 4, 4, 0, 0, 3, 1, 1, 1,
     0, 0, 4, 1, 1, 5,\\  7, 3, 2, 5, 5, 0, 1 ]
     \end{array}
     $
  
  Product $5A$. 

$ \begin{array}{l}
[ 1, 1, 2, 1, 2, 0, 3, 4, 2, 1, 1, 4, 4, 2, 1, 1, 3, 3, 1, 3,\\
 0, 0, 5, 3, 0, 0, 6, 4, 5, 6, 1, 7, 4, 7, 
  0, 0, 3, 8, 2, 6, \\ 11,
   2, 5, 5, 0, 0, 2, 1, 3, 4, 7, 0, 0, 7, 3, 9, 5, 0, 0, 6,\\ 4, 
   2, 13, 6, 0, 4, 4,
   12, 11, 16, 9, 7, 3, 11, 13, 12, 20, 
  5, 10, 6,\\ 11, 
  13, 17, 4, 10, 7, 19, 7, 7, 8, 10, 14, 18,
   19, 5, 10, 12, 23, 7, 12,\\ 6,
    24, 6, 4, 17, 16, 8, 9, 17, 11, 12, 
  23, 8, 24, 18, 26, 21, 29,
   10, 18,\\ 31,
    10, 24, 21, 17, 27, 35, 13, 14, 29, 19, 12, 7, 18, 26, 15, 34, 34, 35, 20,\\ 14,
    36, 14, 39, 8, 29, 
  24, 15, 40, 13, 9, 38, 24, 17, 35, 32, 26, 26, 24, 22,\\ 17,
  31, 39, 29,
   30, 30, 19, 44, 37, 37, 28, 30, 31, 29, 42, 40, 40, 56, 56, 30,\\ 
  30,
  42, 50, 47, 2, 2, 4, 6,
   4, 0, 0, 4, 6, 10, 10, 12, 8, 12, 0, 0,\\ 2,
    4, 16, 10, 0, 0, 2, 12, 10, 0, 0, 0, 0, 0, 
  0, 28, 0, 0, 
  14, 34,\\ 40, 2, 10, 10, 22, 40, 44, 44, 8, 8, 36, 14, 14, 16, 8, 8, 46, 28, 28, 58,\\ 90, 72,
   70, 92, 104, 56, 90 ]
   \end{array}$
   
 Product $6A$. 

$\begin{array}{l}[ 1, 2, 3, 1, 4, 1, 5, 5, 5, 1, 1, 5, 8, 4, 4, 1, 7, 6, 0, 5, \\
0, 0, 10, 7, 0, 0, 10, 6, 6, 13, 3, 14, 
10, 11,   0, 0, 5, 11, 2, 14,\\
 19, 6, 6, 5, 0, 0, 0, 3, 6, 7, 11, 0, 0, 17, 2, 20, 9, 0, 0, 12,\\
  8, 1, 23, 11, 1, 8, 7, 23, 18, 27, 18, 12, 7, 22,  29, 21, 34, 6, 22, 7, \\
  22, 18, 33, 3, 19, 10, 34,
12, 12, 15, 17, 28, 34, 34, 7, 20, 26, 40, 15, 25, \\
3, 40, 9, 6, 34, 25, 18, 21,  30, 21, 18, 
43, 12, 45, 39, 49, 38, 51, 18, 32,\\
 63, 19, 42, 41, 33, 48, 64, 27, 29, 52, 38, 29, 19, 40, 
47, 31, 69, 69, 65, 42,\\
 35, 68,  27, 73, 20, 53, 46, 38, 75, 29, 24, 72, 50, 41, 72, 68, 58, 
52, 54, 50,\\ 44, 64, 75, 58, 69, 65, 49, 85, 75, 75, 63, 68, 65, 63, 90, 87,  83, 118, 118, 74, \\
71, 90, 109, 109, 2, 3, 6, 9, 8, 0, 0, 7, 10, 18, 16, 22, 12, 23, 0, 0,\\
 2, 6, 28, 19, 0, 0, 5, 16, 
18, 0,  0, 0, 0, 0, 0, 52, 1, 1, 26, 59,\\
 76, 11, 18, 18, 39, 77, 80, 77, 22, 22, 66, 27, 27, 33, 
20, 20, 87, 60, 60, 103, \\
175, 148, 152, 187,  215, 140, 201 ]\end{array}$

\end{table}

\section{Elements of 
prime 
order: Fusion, conjugacy and centralizers}
\label{ccls}

\subsection{The character table of $C_\M(t)$ where $t \in 2B$}
First we compute the character table of the full Schur cover of $2^{24}.\Co_1$ from 
which we extract the character table of $C_\M(t)$. For this purpose we use the group $2.\Co_1$
in its $24$-dimensional representation modulo $3$,
and a group $2^{1+24}\Co_1$ constructed in a $4096$-dimensional representation (again over the
field of order $3$) by standard methods. The group whose character table we compute first is the
diagonal product of these two groups, relative to the quotient $\Co_1$,
and is a double cover of the Monster involution centralizer.
The construction of this group is described in detail in \cite{Mconmod3}, and also, with a higher standard of reproducibility, in \cite{Seysen}. 
This can be done without using
any construction of the Monster. See \cite{steps}.

From the character table of $C_\M(t)$ we find $91$ conjugacy classes of elements 
which power to a conjugate of $t$ in $\M$.
Together with the $42$ classes that power into $2A$, and the identity element,
this brings our class count to $42 + 91 + 1 = 134$. 

\subsection{Elements of order $3$ and some $3$-local subgroups}
\label{3elements}
\subsubsection*{\bf Step 1} Elements of classes $3A$ and $3C$ are products of pairs of $2A$ elements. The structure of their
normalizers is $3.\Fi_{24}$ and $S_3 \times \Th$ respectively. To see this we first observe that the centralizers are given in \cite{GMS}.  
The precise structure of $N_\M(\langle x_3\rangle)$, with $x_3 \in 3C$,  follows from the fact that ${\rm Out}(\Th) = 1$. 
Now we need to know that $z$ acts nontrivially on $C_\M(x_1)$ ($x_1 \in 3A$). 
In fact, \cite{GMS} shows that the centralizer of $\langle x_1,z\rangle$ is $\Fi_{23}$ (see Table \ref{suborbits}), 
from which the structure of $N_\M(\langle x_1\rangle)$
follows.

\subsubsection* {\bf Step 2} The group $2.\B$ has two classes of $3$-elements, say $3a$ and $3b$. The class $3a$ fuses to the class $3A$ and 
the other one, represented by $x_2$, has centralizer order $3^{12}|6.\Suz|$. 
We first show that every element of order $3$ in $C_\M(t)$ is conjugate to $x_1$, $x_2$ or $x_3$.
Note that a $2d$-element of $\B$ lifts to a $2^2$ of type $2ABB$ in $\M$, with centralizer of order $2^{26}|O_8^+(2)|$. This centralizer
maps to $2^{1+8}O_8^+(2)$ in the quotient $\Co_1$ of $C_\M(t)$, and contains elements of the three $\Co_1$-classes $3a$, $3b$ and $3c$.
Hence each of these three classes fuses to $3A$ or $3B$. The class $3d$ has centralizer $2^{1+8}A_9$ in $C_\M(t)$, 
and therefore fuses to $3C$.

We now claim that $C_\M(x_2) \cong 3_+^{1+12}.2\Suz$. 
The group $6.\Suz$ is contained in $C_\M(t)$ (where $t \in 2B$) and $O_3(Z(6.\Suz))$ is centralized by a conjugate of $z$.  
Thus we find a copy of $6.\Suz$ in $C_\M(x_2)$. In $C_\M(x_2)$ the involution in $Z(6.\Suz)$ cannot centralize more than $6.\Suz$.
As there are no quasisimple groups with involution centralizer $6.\Suz$ we see that $F^*(C_\M(x_2))$ is not quasisimple and hence 
must be a $3$-group. The involution in the center of $6.\Suz$ must act fixed point freely on $O_3(C_\M(x_2))/\langle x_2 \rangle$
and hence the latter group is elementary abelian. In $2.\B$ we see that $O_3(C_\M(x_2))$ contains a nontrivial extraspecial group which implies 
that the group $C_\M(x_2)$ has shape $3_+^{1+12}.2\Suz$ as claimed. 

\subsubsection* {\bf Step 3} Next we argue that every $3$-element is conjugate to one of $x_i$ with $1 \leq i \leq 3$. 
We first observe (from the character table) that any element of order $3$ in $C_\M(t)$ is either inverted or centralized by a conjugate 
of $z$, 
and thus is conjugate to one of the $x_i$. The elements of order $3$ in $C_\M(z)$ are conjugate to either $x_1$ or $x_2$.
Now, using \cite[Section 4]{W}, we observe that every $3$-element in $3_+^{1+12}.2\Suz.2$ is centralized by an involution.
Since this group contains a full Sylow $3$-subgroup, this concludes the proof that there are exactly three
conjugacy classes of elements of order $3$ in $\M$.

\subsubsection*{\bf Step 4} $\M$ possesses a subgroup $3^8.O_8^-(3)$ which is constructed from two copies of 
$3 \times 3^7.O_7(3)$ which intersect in $3^8.L_4(3)$. The latter group normalizes a $3A$ pure subgroup 
of rank two. [The normalizer of the $3A$ pure rank two subgroup induces a $D_8$ and also normalizes 
the subgroup $3^8$, which proves that the normalizer of the $3^8$ is bigger than $3^8O_7(3)$. 
On the other hand the subgroup $3^8$ contains exactly $1066$ elements from  class $3B$ and must act 
primitively on the set. Inspection of the primitive groups library shows that the automizer 
of the $3^8$ subgroup is (at least) $O_8^-(3)$.] Alternatively we can simply quote  \cite[Theorem 3]{W}.

\subsubsection*{\bf Step 5} Finally we construct the character tables of $3_+^{1+12}.2\Suz.2$, $3.\Fi_{24}$, and  $S_3 \times \Th$.
These 
three tables  will give us all classes of elements (of odd order) which power into a $3X$ element where $X \in \{A,B,C\}$.
There are exactly $31$ classes of such elements.
Together with the 
$134$ classes already found,
this brings our class count to $134+31  = 165$.

\subsection{Elements of order $5$}
\label{5elements}
\subsubsection*{\bf Step 1}
First note that 
$C_\M(z)$ contains two classes of elements of order $5$, with different values of the permutation character.
Hence $\M$ possesses conjugacy classes $5A$ (represented by $y_1$) with 
centralizer $5 \times \HN$ and class $5B$ (represented by $y_2$) with centralizer order $5^{1+6}|2.\Ja_2|$. 
As already noted in Section~\ref{3elements}, 
elements of $\B$-class $2d$ lift to a $2^2$ of type $2ABB$ in $\M$, with centralizer of order $2^{26}|O_8^+(2)|$.
This centralizer maps to a subgroup $2^{1+8}O_8^+(2)$ in the quotient $\Co_1$ of $C_\M(t)$. Since $2^{1+8}O_8^+(2)$
intersects all three $\Co_1$-classes of elements of order $5$, it follows that every $5$-element in $C_\M(t)$ is in one of the
two $\M$-classes $5A$ or $5B$.

We deduce that the structure 
of the centralizer of a $5B$-element is $5^{1+6}2\Ja_2$ in exactly the same way as for the case $C_\M(x_2)$. 
More precisely we first note that we can see $5 \times 2\Ja_2$ inside $C_\M(t)$. Then we observe that 
there are no simple groups with involution centralizer $C$ having a composition factor isomorphic to $\Ja_2$. Thus we deduce that 
$|O_5(C_\M(y_2))| = 5^7$. Since $Z(2.\Ja_2)$ inverts every element in $O_5(C_\M(y_2))/\langle y_2 \rangle$ and $2.\Ja_2$ acts irreducibly
we deduce that $O_5(C_\M(y_2))$ is extraspecial of exponent $5$. 

\subsubsection*{\bf Step 2} We show that every $5$-element is conjugate to one of the $y_i$ by showing that every $5$-element in $C_\M(y_2)$ 
is centralized by an involution.  This calculation was done in \cite[Section 9]{W}  without using the stated assumption that
$\M$ contains just two classes of elements of order $5$, since it is a calculation inside
the group $5^{1+6}.2\Ja_2$ that does not depend on the existence of the Monster. 
However, we also repeated the calculation as part of the next step.

\subsubsection* {\bf Step 3} The character tables of the centralizer and normalizer of $5A$ are easy to deduce from the already verified 
character table of $\HN$. We (re-)calculate the character tables of the centralizer and normalizer of $5B$, using Magma.
It is then straightforward to
 deduce that $\M$ possesses $8$ classes of order coprime to $6$ that power up to $5A$ or $5B$.

Our class count now stands at $165 + 8 = 173$.

\subsection{Elements of order $11$} The group $C_\M(z)$ contains a rational element of order $11$ with centralizer of order 
$11|\Ma_{12}|$. We call this class $11A$. 
Next we observe that if $S \in {\rm Syl}_{11}(\M)$ then $S$ is elementary abelian with $C_\M(S) = S$. 
Now we also see, from the facts that the $11A$ element is rational and that $\Ma_{12}$ contains $11{:}5$, that 
$5^2{:}2 \leq N_\M(S)$. Thus there exists an element $f$ of order $5$ normalizing $11A$ and centralizing an element from $S$ not contained 
in the cyclic group normalized by $f$. 

On the other hand $f$ lies in a conjugate of $C_\M(z)$ which implies that 
$f$ is from class $5A$ which in turn implies that $C_S(f)$ fuses to $11A$. As fusion of $11$-elements is controlled 
in the normalizer of $S$ we see that $|N_\M(S)| > 11^2.5^2.2$. On the other hand $N_\M(S)/S$ is a subgroup of 
${\rm GL}_2(11)$ of order not divisible by $11$. This leaves exactly $5$ possibilities for the order of 
$N_\M(S)$ only one of which is compatible with Sylow's theorem. Thus we see that $N_\M(S)/S \cong (5 \times 2.A_5)$
and that all $11$-elements are conjugate in $\M$.

Our class count now stands at $173 + 1 = 174$.

\subsection{Elements of orders $17$, $19$, $23$, $31$, and $47$}
The elements lie in cyclic Sylow subgroups which fuse into $2.\B$. So 
we compute the centralizer orders from the permutation character of $\M$ on the cosets of $C_\M(z) \cong 2.\B$.

The elements of order $17$ and $19$ are already rational in $2.\B$ whereas the elements of order $23$, $31$ and $47$ are 
normalized by elements of order $11$, $15$ and $23$ respectively. 
Using Sylow's theorem we now see that the normalizer orders of the $17$, $19$, $23$, $31$, and $47$-elements are 
as claimed in the $\mathbb{ATLAS}$. 

Our class count now stands at $174 + 1 + 1 + 2 + 2 + 2 = 182$. 

\subsection{$13$-Elements} 
\label{13elements}
\subsubsection* {\bf Step 1} One class of $13$-elements fuses into $C_\M(z)$, is rational, and has centralizer order $13|L_3(3)|$
(from the permutation character). 
Call this class of elements $13A$. 
If $x\in 13A$ then $C_\M(x)/\langle x\rangle$ is a group with involution centralizer $2S_4$, of index $3^2.13$, so $C_\M(x)\cong 13\times L_3(3)$.
The class $13A$ is not $13$-central since its centralizer order is not divisible by $13^3$. 
This implies that the Sylow $13$-subgroups of $\M$ are not abelian. Thus \cite[Proposition 10.4]{GLS} implies that 
the Sylow $13$ subgroups of $\M$ are extraspecial. As $\M$ is simple, \cite[Proposition 15.21]{GLS} implies that
the Sylow $13$ subgroups must be of exponent $13$, that is isomorphic to $13^{1+2}_+$. 

\subsubsection*{\bf Step 2}
The centralizer of a Sylow $13$ subgroup in ${13} \times L_3(3)$ is ${13}\times {13}$ which implies that $C_\M(P_{13}) \leq P_{13}$
for $P_{13} \in {\rm Syl}_{13}(\M)$. Let $A$ be a Sylow $13$-subgroup of $C_\M(x)\cong 13 \times L_3(3)$.

As $P_{13}$ is extraspecial, $A$ contains a unique cyclic $13$-central subgroup (in $P_{13}$). [Uniqueness follows from the fact that $A \lhd P_{13}$
and so there is a $13$-element $d$ in $P_{13} \setminus A$ acting on the $1$-spaces of $A$ with orbit lengths $13$ and $1$.
The elements in the group fixed by $d$ are central in $P_{13}$, while the rest are in class $13A$.] 
Every Sylow $13$-subgroup in the centralizer of a $13A$-element is elementary abelian of rank $2$ and 
is normalized by a group of order $36$. Thus the unique subgroup of 
$13$-central elements is also normalized 
by a group of order $36$. 

Applying Sylow's theorem and observing that $N_\M(P_{13})/P_{13}$ is isomorphic to a subgroup of ${\rm GL}_2(13)$ 
of order divisible by $36$ yields that $|N_\M(P_{13})| = 288\times13^3$. 
There are two
conjugacy classes of subgroups 
of order $288$ in ${\rm GL}_2(13)$, one with orbit sizes $6+8$ on the $14$ points of the projective line,
and the other with orbits $2+12$.  Since we know there is an orbit of length $8$, the Sylow $13$-normalizer
has structure $13^{1+2}{:}(3\times 4S_4)$.

\subsubsection*{\bf Step 3}
The structure of the normalizer of $P_{13}$ also yields that $13$-central elements are rational and centralized 
by an involution. We call this class $13B$. The involution centralizers $C_\M(z)$ and $C_\M(t)$ each contain 
a unique class of $13$ elements and $C_\M(z)$ contains elements from class $13A$. Thus $C_\M(t)$ must contain 
only $13$-elements from class $13B$. Thus it follows that the involution centralizers of $13B$-elements of $\M$
are double covers of $A_4$ and hence have quaternion Sylow $2$-subgroups. Thus the Brauer--Suzuki theorem 
implies that $C_\M(13B)/Z^*(C_M(13B))$ is a group with Sylow $2$-subgroup isomorphic to $2 \times 2$;
hence has a normal $2$-complement or is isomorphic to $L_2(q)$ with $q \equiv 3,4 \ \mbox{mod} \ 8$. 

The 
cases $L_2(q)$ with $q > 5$ are impossible. To see this note that on the one hand $P_{13}$ is non-abelian and 
on the other hand  $13$ does not divide the Schur multiplier of $L_2(q)$ for any $q$ and $13$ divides
${\rm Out}(L_2(q))$ if and only if $q = q_0^{13}$.
Also $Z^*(C_\M(13B)) = O_{2'}(C_\M(13B)){:}2$. In particular this shows that $O_{13}(C_\M(13B))$ is characteristic 
in $C_\M(13B)$ and of order $13^3$; i.e. $C_\M(13B)$ is conjugate to a subgroup of $N_\M(P_{13})$. 
Thus we deduce that $C_\M(13B) = 13^{1+2}{:}Q_8{:}3$.

\subsubsection*{\bf Step 4}
Now we observe that $N_\M(P_{13})$ has two orbits on the $1$-spaces of $P_{13}/Z(P_{13})$;
one of length $8$ and one of length $6$. In particular, every non-trivial element of $P_{13}$ is centralized
by an involution, so is in class $13A$ or $13B$.
Thus we have found two classes of $13$ elements and their centralizer orders. 
Since these orders are divisible only by the primes $2$, $3$ and $13$, there are no further elements
of order divisible by $13$.
Our class count now stands at $182 + 2 = 184$.

\subsection{Elements of orders $29$, $41$, $59$ and $71$}
\label{29-71elements}
\subsubsection* {\bf Step 1} $C_\M(x_1)$ contains a Sylow-$29$ subgroup $P_{29}$ of $\M$. 
In $N_\M(\langle x_1 \rangle)$ the normalizer of $P_{29}$ is of the form $(29{:}14 \times 3).2$.
We see that all $29$-elements are conjugate and thus rational. Therefore the normalizer can grow only 
via the centralizer. We note that the only primes which could potentially grow the centralizer order
beyond the $3 \times 29$ that we see in $|C_\M(x_1)|$ are $7$, 
$41$, $59$ and $71$. Using 
Sylow's theorem we deduce that the only possible orders for $|C_\M(P_{29})/P_{29}|$ are 
$3$ and $3.59$. 

We rule out the possibility that $|C_\M(P_{29})/P_{29}|$ has order divisible by $59$ by observing that 
in all cases the Sylow $59$-subgroup 
of $C_\M(P_{29})/P_{29}$ is normal and hence normalized by a $3$ element. As $58$ is not divisible by 
$3$ we see that $59$ must divide the centralizer order of a $3$-element; a contradiction. 

Thus we see that $P_{29}$ is rational and that $|C_\M(P_{29})| = 3.29$.

\subsubsection* {\bf Step 2}  $3^8.O_8^-(3)$ contains a Sylow $41$-subgroup $P_{41}$ of $\M$ which is normalized by 
an element of order $4$. Thus the only possible further divisors of $|N_\M(P_{41})/P_{41}|$ must be divisors of 
$2.5.7^6.59.71$. Sylow's theorem allows for the possibilities $2.5$ and $2.7.71$. In the latter case 
the Sylow $71$ subgroup of $N_\M(P_{41})$ is normal and hence normalized by an element of order $4$. 
This forces $71$ to divide the order of an involution centralizer; again a contradiction. Thus 
$N_\M(P_{41}) = P_{41} \rtimes Z_{40}$, as $41$ does not divide the order of any $5$-centralizer. 
 
\subsubsection*{\bf Step 3} The only possible divisors of $|N_\M(P_{59})/P_{59}|$ are divisors of 
$2.29.7^6.71$. 
Sylow's theorem 
implies that $|N_\M(P_{59})/P_{59}| =29$ and thus 
$N_\M(P_{59}) = P_{59} \rtimes Z_{29}$.

\subsubsection*{\bf Step 4} 
The only possible divisors of $|N_\M(P_{71})/P_{71}|$ 
are divisors of $2.5.7^5$. From Sylow's theorem we infer that the only possible divisor is $35$ and 
thus we have $N_\M(P_{71}) = P_{71} \rtimes Z_{35}$, again as $71$ does not divide the centralizer of any $3$ or $5$-element. 

Our conjugacy class count now stands at $184 + 1 + 1 + 2 + 2 = 190$.

\subsection{$7$-Elements} 
\label{7elements}

\subsubsection* {\bf Step 1}
$C_\M(z)$ possesses a class of $7$-elements $7A$ with centralizer $7 \times \He$; again the 
order of the centralizer can be deduced from the permutation character of $\M$ on the cosets of $C_\M(z)$ whereas 
the structure can be deduced from the involution centralizer and a $2$-local characterization of $\He$.
The group $C_\M(x_1)$ contains two classes of $7$-elements which are not fused in $\M$ as they have distinct 
traces on the $196883$-dimensional module of $\M$ (see Lemma~\ref{3F24char}). 
One of the classes fuses to $7A$ and 
the other, call it $7B$,  fuses to the $7$-central class of $\He$ and also into $C_\M(t)$. From this we deduce 
that the centralizer of a $7B$ element contains an extraspecial $7$-group of order $7^3$ and a subgroup $7 \times 2.A_7$.

\subsubsection*{\bf Step 2} By considering the centralizer of a $7A$-element we see that the
centralizer of a Sylow $7$-subgroup is contained in its center and that the order 
of the center is at most $7$; i.e., generated by a $7B$-element. As $7B$ is rational and is centralized 
by $2.A_7$ we see that  $|N_\M(P_7)/P_7|$ is divisible by $36$ for every $P_7 \in {\rm Syl}_7(\M)$.

\subsubsection*{\bf Step 3}
There exist subgroups of rank $2$ which are $7A$ pure (in $C_\M(z)$), which means that the $7$ share of 
the centralizer of a prime order element of order $>5$ centralizing a $7B$ element is at most $7^4$.

We have found two classes of $7$-elements and two classes of elements of order $7.17 = 119$.
Our class count stands at $190 +4  = 194$.

\subsubsection*{\bf Step 4}
\label{7Belements} Finally we determine the order of the $7B$ centralizer.
We begin by observing that our previous subsections show that no prime $\geq 11$ divides $|C_\M(7B)|$.
The group $\HN$ is contained in $C_\M(z)$ and thus contains only $7A$ elements.
Conversely, the $7A$ centralizer $7\times \He$ contains only one class of elements of order $5$.
Hence the centralizer of a $5B$-element,
i.e. $5^{1+6}:2.J_2$, contains elements of class $7B$. So we see that $7B$ is centralized by a $5B$ element and that 
$|C_\M(7B)|_5 = 5$. 

\subsubsection*{\bf Step 5}
Next we observe that $7B$ is contained in $\He \leq 3.\Fi_{24} = C_\M(x_1)$ and in 
$C_\M(x_2)$ but not in $C_\M(x_3)$ as $\Th < C_\M(z)$. From this we see that 
$|C_\M(7B)|_3 = 9$.
The fact that $7B$ is contained in $C_\M(t)$ but not $C_\M(z)$ shows that 
$|C_\M(7B)|_2 = 2^4$. So we have now established that $|C_\M(7B)| \leq 7^5|2.A_7|$.

Combining the information from the preceding subsections we have accounted for all but 
$|\M|/(7^5|2.A_7|)$ elements of $\M$.
Hence
the conjugates of $7B$ now account for all the remaining elements of $\M$.

\subsubsection*{\bf Step 6}
As there does not exist a quasisimple group of order $7^a|2.A_7|$ with $1 \leq a \leq 5$ we see 
that $|O_7(C_\M(7B))| = 7^5$. Now $2.A_7$ must act nontrivially on $O_7(C_\M(7B))$.
As the minimal dimensional module for $2.A_7$ over ${\rm GF}(7)$ is $4$, 
we deduce that $O_7(C_\M(7B))$ is extraspecial of order $7^5$ and exponent $7$. 

\subsection{The conjugacy class list}
In summary we have established
\begin{prop}
The group $\M$ has exactly $194$ conjugacy classes of elements. The orders 
and centralizer structure are as given in the $\mathbb{ATLAS}$.
\end{prop}

This concludes our analysis of conjugacy classes and centralizers.

\section{Computation of the character table}
\label{chartab}
\subsection{Conjugacy class fusion from certain subgroups and character values on $\chi$}
By Hypotheses $\M$ possesses an ordinary character $\chi$ of degree $196883$. 

We assume now that the character tables of $C_\M(z)$, $C_\M(t)$, $C_\M(x_1) = 3.\Fi_{24}$,
$C_\M(x_2)$, $C_\M(x_3)$, $C_\M(y_1)= C_\M(5A)$, $C_\M(5B)$ and $C_\M(7A)$ have been computed. 
We know, for the 194 classes, the element orders, centralizer orders,
    and partial fusions from the character tables of subgroups.

 Next we determine the $p$-th power maps, for all primes $p$ up to the
    maximal element order, which will be needed later for inducing
    characters from cyclic subgroups.
    For that, it is sufficient to use the following ingredients:
\begin{itemize}
  \item the commutativity of subgroup fusions and power maps;
      GAP provides functions for that, and this way, many more values
      of the class fusions get determined than had been known from the
      definitions of the classes,
\item the fact that the character values on the classes of element orders
      59, 71, and 119 lie in a known quadratic field,
\item a few allowed choices in the 2nd and 3rd power map, for example
      concerning element orders 23 and 46,
\item one argument about a class of element order 21, which uses the
      structure $7\times \He$ of the 7A centralizer (the explicit character
      table is not used besides this argument).
\end{itemize}

 Next we compute the values of the 199883 on all classes.
\begin{itemize}
 \item The decompositions of the restrictions to $2.\B$ and $3.\Fi_{24}$ are known,
      and the partial class fusions give us many values.
\item From these values and the partial fusion from $C_\M(2B)$,
      the decomposition of the restriction to $C_\M(2B)$ is determined,
      the matrix of possible constituents on the known classes has full rank.
\item  All missing values except one are determined by their congruence
      modulo $p$ to the value at $p$-th powers, and the fact that the square
      of the value must be smaller than the centralizer order.
\item The missing value is computed from the scalar product with the
      trivial character.
      \end{itemize}
      
If $H$ is one of the centralizers whose character table we have computed (or is known) 
then we can deduce $\chi|_H$. In this way we can compute the character values of 
all elements in $\M$ which fuse into one of the centralizers above. 

Thus we need to determine the values of $\chi$ on classes $41A$, $59AB$ or $71AB$. 
Since $196883$ is divisible by $59$ and $71$, this character forms a block of defect $0$
for both primes, so the character values are $0$ on these four classes.
Similarly, $196883 \equiv 1 \bmod 41$, so that this character lies in a block of defect $1$ for $p=41$,
and since the Sylow $41$-subgroup is self-centralizing the character value on class $41A$ is $1$.

   Here we can decide which of the two possible factors of the
    $3^{1+12}.6\Suz.2$ really occurs, because only one of the two tables
    admits a decomposition of the 196883.

We have now verified that the values of $\chi$ are as claimed in the $\mathbb{ATLAS}$.

\subsection{Verifying the characters}
From our earlier analysis we have already computed the values of the character $1_{2.\B}^\M$ and 
we know the exact character values of $\chi$. 
From this point on, the verification of the full character table is more or less straightforward.

  Now we take the character table head of $\M$ (including all power maps),
    the irreducible character of degree 196883,
    the character tables of $2.\B$, $3.\Fi_{24}$, $2^{1+24}.\Co_1$,
    and $3^{1+12}.2.\Suz.2$.
    The class fusions of these subgroups are determined by the given data,
    we can induce from them, and from cyclic subgroups.
    This suffices to verify the irreducibles; here the \Atlas\ table of $\M$ is
    used as an ``oracle'', in the same way as the \Atlas\ table of $\B$ was
    used in its verification.

\end{document}